\documentclass[12pt,reqno]{amsart}
\usepackage{amssymb,amsmath,amsthm,amscd,latexsym,amsfonts}

\usepackage{xy}
\xyoption{all}
\usepackage{mathtools}
\usepackage[T1]{fontenc}
\usepackage{graphicx}
\usepackage{color}
\usepackage{enumitem}
\usepackage{hyperref}

\usepackage[a4paper,top=3cm, bottom=3cm, left=3cm, right=3cm]{geometry}

\newtheorem{theorem}[subsection]{Theorem}
\newtheorem{lemma}[subsection]{Lemma}
\newtheorem{proposition}[subsection]{Proposition}
\newtheorem{corollary}[subsection]{Corollary}

\theoremstyle{definition}
\newtheorem{remark}[subsection]{Remark}
\newtheorem{definition}[subsection]{Definition}

\numberwithin{equation}{section}

\newcommand{\al}{\alpha}
\newcommand{\lam}{\lambda}

\DeclareMathOperator{\Ker}{\mathsf {Ker}}

\DeclareMathOperator{\Img}{\mathsf {Im}}
\DeclareMathOperator{\id}{id}

\begin{document}
	\title[Some generalisations of Schur's Theorem]{Some generalisations of Schur's and Baer's theorem and their connection with homological algebra}
	
	\author{Guram Donadze}
	\address[Guram Donadze]{Department of Mathematics, University of Brasilia, Brasilia-DF,70910-900 Brazil
	\newline and \newline
		Institute of Cybernetics of the Georgian Technical University, Sandro Euli Str. 5, 0186, Tbilisi, Georgia}
	\email{gdonad@gmail.com}

	\author{Xabier Garc\'ia-Mart\'inez}
\address[Xabier Garc\'ia-Mart\'inez]{Departamento de Matem\'aticas, Esc.\ Sup.\ de Enx.\ Inform\'atica, Campus de Ourense, Universidade de Vigo, E-32004, Ourense, Spain
	\newline and\newline
	Faculty of Engineering, Vrije Universiteit Brussel, Pleinlaan 2, B-1050 Brussel, Belgium}
\email{xabier.garcia.martinez@uvigo.gal}

	\thanks{This work was supported by Ministerio de Econom\'ia y Competitividad (Spain), with grant number MTM2016-79661-P. The second author is a Postdoctoral Fellow of the Research Foundation--Flanders (FWO)}

\begin{abstract}
	Schur's Theorem and its generalisation, Baer's Theorem, are distinguished results in group theory, 
	connecting the upper central quotients with the lower central series. 
	The aim of this paper is to generalise these results in two different directions, using novel methods related with the non-abelian tensor product. 
	In particular, we prove a version of Schur-Baer Theorem for finitely generated groups.
	Then, we apply these newly obtained results to describe the $k$-nilpotent multiplier, for $k\geq 2$, and other invariants of groups.
\end{abstract}



\maketitle

\section{Introduction}\label{Intr}

Given a group $G$, Schur~\cite{Schur} was one of the first authors that found a relation between the central factor group $G/Z(G)$ and the derived subgroup $G'$. In fact, the renowned theorem that carries his name says that, if $G/C$ is a finite group, where~$C \leq Z(G)$, then $G'$ is finite. Even though the authorship of this result is a bit fuzzy (see~\cite{DKP} for details), it is commonly known as Schur's Theorem. It has been actively studied and generalised in many different directions. For instance, if the theorem still holds replacing the class of finite groups by some other one, we say that it is a \emph{Schur class}.

Baer~\cite{Baer} proved a broader version of Schur's Theorem stating that if for a natural number $k$ the quotient $G/Z_{k}(G)$ is finite, then $\gamma_{k+1}(G)$ is finite. This result is commonly known as Baer's Theorem. The aim of the present paper is to give two generalisations of Baer's Theorem involving group extensions. To do so, we make use of the non-abelian tensor product introduced by Brown and Loday~\cite{BL, BrLo} in the context of homotopy theory. 

The first direction of our generalisations relates different properties of a group $G$ with the properties
of $\gamma_{n+1}(H)$, for all extensions $p\colon H \to G$ where the kernel of $p$ is inside the $n$-th
centre of $H$ (see Proposition~\ref{prop1} and Proposition~\ref{prop2}). On the other hand, the second direction of our generalisations connects the properties
of a group $G$ with the properties of the derived series of $H$, for all extensions $p\colon H \to G$ where
the kernel has some special property related with commutators (see Proposition~\ref{prop3} and Proposition~\ref{prop4}).

An interesting application of these results becomes visible when we apply them to the free presentation of a group. In this case, we are able to deduce properties of the presentation by the properties of the original group, with applications to non-abelian 
homological algebra.

The paper is organised as follows. In Section~\ref{S:nonabelian} we recall some important definitions and results concerning the non-abelian tensor product that will be used later. In Section~\ref{S:Baer} we state and prove the two generalisations of Baer's Theorem. Finally, in Section~\ref{S:Homological} we apply the obtained results to describe the left derived functors of certain
class of group valued functors. In particular, we get an effective description of the $k$-nilpotent multipliers, for $k\geq 2$,
and other invariants of groups in terms of the non-abelian tensor product.  

\section{The non-abelian tensor product}\label{S:nonabelian}

We begin by recalling some important definitions and stating some known results that will be useful in the upcoming sections.

First of all, we should fix some notation. Given a group G, by $\gamma_n(G)$
and $Z_n(G)$ we mean the $n$-th term of the lower central series and upper central series, respectively.
Moreover, we denote by $\Gamma_n(G)$ the $(n-1)$-th term of the derived series, i.e., 
${\Gamma_1(G)=G}$ and $\Gamma_{n+1}(G)=[ \Gamma_n(G) , \Gamma_n(G)]$, for $n\geq 1$.  
By the commutator of two elements $[g, h]$ we mean $ghg^{-1}h^{-1}$.
Note that we write the conjugation operation on the left, so $^gg' = gg'g^{-1}$, whenever~$g, g'\in G$. 
In fact, we are considering left actions in general.

The non-abelian tensor product of groups is defined for a pair of groups that act on each other. A standard requirement for these actions is some sense of \emph{compatibility}~\cite{BL}.

\begin{definition} Let $G$ and $H$ be groups acting on each other. The mutual actions are said to be \emph{compatible} if
	\[
	^{(^hg)}h' = \; ^{h}(^{g}(^{h^{-1}} h'))  \; \text{and}\; ^{(^gh)}g' = \; ^{g}(^{h}(^{g^{-1}} g')),
	\]
	for each $g, g' \in G$ and $h, h' \in H$. 
\end{definition}

\begin{definition} Let $G$ and $H$ be two groups that act compatibly on each other. Then the \emph{non-abelian tensor product}
	$G\otimes H$ is the group generated by the symbols $g\otimes h$ for $g\in G$ and $h\in H$ with relations
	\begin{align*}
		gg' \otimes h &= (^gg' \otimes \;^gh)(g\otimes h), \\
		g\otimes hh' &= (g\otimes h)(^hg \otimes \;^hh'),
	\end{align*}
	for each $g, g' \in G$ and $h, h' \in H$. 
\end{definition}

The special case where $G = H$ and all actions are given by conjugation, is called the \emph{tensor square} $G\otimes G$.  

\begin{definition}\label{D:exterior} Let $G$ be a group acting on itself via conjugation. The \emph{non-abelian exterior square} $G\wedge G$ is the group generated by the symbols 
	$g\wedge h$ for $g, h\in G$ with relations
	\begin{align*}
		gg' \wedge h &= (^gg' \wedge \;^gh)(g\wedge h), \\
		g\wedge hh' &= (g\wedge h)(^hg \wedge \;^hh'), \\
		g\wedge g &= 1,
	\end{align*}
	for each $g, g', h, h' \in G$. 
\end{definition}
A more general construction of the non-abelian exterior product for arbitrary crossed modules can be found in~\cite{BrLo}.

We have the homomorphism $\mu^G\colon G\wedge G \to G$ given by $g\wedge h \mapsto [g, h]$, for every~$g, h\in G$. The following theorem shows the importance of this homomorphism in homology theory (see~\cite{Ellis2} and \cite{Miller}).

\begin{theorem}\label{thm4} Let $H_2(G)$ denote the second homology group with integer coefficients. Then
	\[
	H_2(G) \cong \Ker \big\{ \mu^G \colon G\wedge G \to G \big\}.
	\]
\end{theorem}

Let us also recall the behaviour of the non-abelian tensor square with short exact sequences.

\begin{theorem}[\cite{BL}]\label{thm1} Let $1\to N \to H \to G \to 1$ be an extension of groups. Then there is
	an exact sequence $1 \to \overline {N} \to H\otimes H \to G\otimes G \to 1$, where $\overline{N}$ is the normal subgroup 
	of $H\otimes H$ generated by the elements of the form $h\otimes x$ and $x\otimes h$ for each~$h\in H$ and
	$x\in N$.
\end{theorem}

The non-abelian tensor product preserves some interesting properties. We present some of them as an example (see~\cite{Ellis, Mor}):

\begin{theorem}\label{thm2} Let $G$ and $H$ be groups acting compatibly on each other.
	\begin{enumerate}
		\item[\rm (i)] if $G$ and $H$ are finite (or $p$-groups), then $G\otimes H$ is finite (or a $p$-group);
		\item[\rm (ii)] if $G$ and $H$ are perfect groups, then $G \otimes H$ is perfect;
		\item[\rm (iii)] if $G$ and $H$ are polycyclic, then $G\otimes H$ is polycyclic. 
	\end{enumerate}
\end{theorem}

In order to study the preservation of finitely generated subgroups, we will need the notion of derivative~\cite{Visscher}.

\begin{definition} Let $G$ and $H$ be groups with $H$ acting of $G$. The \emph{derivative} of $G$ by $H$ is the subgroup of $G$ defined by 
	\[
	D_H(G)=\big\langle g\; ^h\!g^{-1} \mid g\in G, h\in H\big\rangle.
	\]
\end{definition}

\begin{theorem}[\cite{DLT}]\label{thm3} Let $G$ and $H$ be finitely generated groups acting compatibly on each other. 
	Then $G\otimes H$ is finitely generated if and only if $D_G(H)$ and $D_H(G)$ are finitely generated.
\end{theorem}

\section{Baer's Theorem and other generalisations of Schur's Theorem}\label{S:Baer}

Let $G$ be a group and $G \otimes G$ be its tensor square. There is a well defined action of $G$ on $G \otimes G$ by 
\[
^{g_3}(g_1\otimes g_2)= \;^{g_3}g_1\otimes \;^{g_3}g_2,
\]
and there is also a well defined action of $G \otimes G$ on $G$ given by
\[
^{g_1\otimes g_2}g_3=\;^{[g_1, g_2]}g_3.
\]
In this way, we can define the non-abelian tensor product $(G \otimes G) \otimes G$, denoted by~$G^{\otimes 3}$. Furthermore, for any $n \geq 3$, we can inductively define the \emph{$n$-fold tensor product}, denoted by $G^{\otimes n}$, by considering the actions of $G$ and $G^{\otimes n-1}$ on each other defined by
\begin{align*}
	^{g_{n}}\big(\cdots ((g_1\otimes g_2)\otimes g_3) \otimes \cdots \otimes g_{n-1}\big) &=
	\big(\cdots ((^{g_{n}}g_1\otimes \;^{g_{n}}g_2)\otimes \;^{g_{n}}g_3) \otimes \cdots 
	\otimes \;^{g_{n}}g_{n-1}\big), \\
	^{(\cdots ((g_1\otimes g_2)\otimes g_3) \otimes \cdots \otimes g_{n-1})}g_{n} &=\;
	^{[\cdots [[g_1, g_2], g_3] , \cdots , g_{n-1}]}g_{n}.
\end{align*}
Note that all actions are compatible. See~\cite{Ellis3} for more details.

Besides that there is a well-defined homomorphism $\lam_n ^G \colon  G^{\otimes n} \to G$ defined on generators by 
\[
(\cdots ((g_1\otimes g_2)\otimes g_3) \otimes \cdots \otimes g_{n}) \mapsto [\cdots [[g_1,  g_2], g_3], \cdots , g_{n}].
\]
In particular, we have that
\[
\lam^{G}_{n+1}((\cdots ((g_1\otimes g_2)\otimes g_3) \otimes \cdots \otimes g_{n+1})) = 1, 
\]
whenever $g_i \in Z_n(G)$ for some $1 \leq i \leq n+1$. 

\begin{lemma} \label{lemma1} Let $1\to N \to H \to G \to 1$ be an extension of groups such that ${N\leq Z_n(H)}$ for a fixed 
	positive integer $n$. Then, there exists a surjective homomorphism $ G^{\otimes n+1} \to \gamma_{n+1}(H)$ making
	the following diagram commutative: 
	\[
	\xymatrix{
		G^{\otimes n+1} \ar[r]^{\id}\ar[d] & G^{\otimes n+1} \ar[d]^{\lam_{n+1}^G}\\
		\gamma_{n+1}(H) \ar[r]^{}&  G  }
	\]
\end{lemma}

\begin{proof} By Theorem \ref{thm1} we have an exact sequence of groups
	\[
	1 \to \overset{n+1}{\underset{i=1}{\prod}}N_i \to H^{\otimes n+1} \to G^{\otimes n+1} \to 1,
	\]
	where $N_i$ is the normal subgroup of $ H^{\otimes n+1}$ generated by all the elements of the form 
	$(\cdots ((h_1\otimes h_2)\otimes h_3)\otimes \cdots \otimes h_{n+1})$ with $h_i\in N$ and $h_j\in H$ for $j\neq i$.
	Moreover, we have the following commutative diagram:
	\[
	\xymatrix{
		H^{\otimes n+1} \ar[r]^{}\ar[d]_{\lam_{n+1}^H} & G^{\otimes n+1} \ar[d]^{\lam_{n+1}^G}\\
		H \ar[r]^{}&  G  }
	\]
	Clearly, $\Img (\lam_{n+1}^H ) = \gamma_{n+1}(H)$. Then, since $N\leq Z_n(H)$, we have that $\lam^G_{n+1}(N_i)=1$
	for each $i=1, \ldots, n+1$. Therefore, we have obtained an epimorphism 
	$G^{\otimes n+1} \to \gamma_{n+1}(H)$, induced by $\lam_{n+1}^H$. 
\end{proof}

Let us now obtain one of the generalisations of Baer's Theorem.

\begin{proposition} \label{prop1} Let $1\to N \to H \to G \to 1$ be an extension of groups such that $N\leq Z_n(H)$ for a fixed 
	positive integer $n$. We have that,
	
	\begin{enumerate}
		\item[\rm (i)]   if $G$ is finite (or a $p$-group), then $\gamma_{n+1}(H)$ is finite (or a $p$-group);
		\item[\rm (ii)]  if $G$ is a perfect group, then  $\gamma_{n+1}(H)$ is a perfect group;
		\item[\rm (iii)] if $G$ is polycyclic, then $\gamma_{n+1}(H)$ is polycyclic.
	\end{enumerate}
\end{proposition}

\begin{proof}  By Lemma \ref{lemma1} it suffices to show that $G^{\otimes n+1} $ is a finite group
	(or a $p$-group, or a perfect group, or a polycyclic group). Assume that $G$ is a finite group (or a $p$-group, or a perfect group, or a polycyclic group). 
	Recall that all mutual actions involved into the construction of $G^{\otimes n+1}$ are compatible. 
	Therefore, by Theorem \ref{thm2} we inductively see that the groups 
	$G^{\otimes 2} , G^{\otimes 3}, \ldots, G^{\otimes n+1} $ are finite groups (or $p$-groups, or perfect groups, or polycyclic groups). 
\end{proof}

Now we need a technical result to prove the result in the finitely generated case, which was proved in~\cite{DLP} for~$n=1$.

Let $A$ and $B$ be two groups acting compatibly on each other. Since the derivative~$D_B(A)$ is closed under the action of $B$, 
we can form $D_B(D_B(A))$. By the same argument  one can form $D_B(D_B(D_B(A)))$ and so on. Let us set $D^0_B(A)=A$ 
and $D^k_B(A)=D_B(D^{k-1}_B(A))$ for $k\geq 1$.
Consider the non-abelian tensor product  $D^k_B(A)\otimes B$, for each $k\geq 0$. Since  $D^k_B(A)$ and $B$
are acting compatibly on each other, there is a well-defined action of $B$ on  $D^k_B(A)\otimes B$ given by
\[
^b(x\otimes c) = \; ^bx \otimes \; ^bc,
\]
for each $x\in D^k_B(A)$ and $b, c \in B$. Consequently, for each $k, l\geq 0$, we can consider the group $D_B^l (D^k_B(A)\otimes B)$.
\begin{lemma}\label{L:tech}
	We have that
	\[
	D_B^{l}(D_B^{k}(A) \otimes B) = \Img \Big\{ D_B^{k+l}(A) \otimes B  \to D_B^{k}(A) \otimes B \Big\},
	\]
	where $D_B^{k+l}(A) \otimes B  \to D_B^{k}(A) \otimes B$ is the homomorphism induced by the natural inclusion
	$D^{k+l}_B(A) \to D^{k}_B(A)$ for each $k, l \geq 1$.
\end{lemma}

\begin{proof}
	By~\cite{BJR}, for each $a\in A$ and $b, c \in B$ we have the following identity:
	\[
	(a\otimes b)\:^c(a\otimes b)^{-1}  = a^ba^{-1} \otimes c,
	\]
	which implies that 
	\[
	D_B(D_B^{k}(A) \otimes B) = \Img \Big\{ D_B^{k+1}(A) \otimes B  \to D_B^{k}(A) \otimes B \Big\},
	\]
	where $D_B^{k+1}(A) \otimes B  \to D_B^{k}(A) \otimes B$ is the homomorphism induced by the natural inclusion
	$D^{k+1}_B(A) \to D^{k}_B(A)$. Therefore, by induction we get the desired result.
\end{proof}

%

\begin{proposition}\label{prop2} Let $G$ be a finitely generated group. Then the following are equivalent, for all $n \geq 0$:
	\begin{itemize} 
		\item[\rm(i)] $\gamma_{n+1}(G)$ is a finitely generated group;
		
		\item[\rm(ii)]$\gamma_{n+1}(H)$ is a finitely generated group for any extension of groups $1 \to N \to H \to G \to 1$ with $N\leq Z_n (H)$. 
	\end{itemize}
\end{proposition}
\begin{proof} The implication (ii) $\implies$ (i) is obvious, so let us show the converse.  
	By Lemma \ref{lemma1} it suffices to show that $G^{\otimes n+1}$ is finitely generated. 
	We will prove that~$G^{\otimes k}$ is finitely generated for each $1\leq k \leq n$, where, to ease notation, we assume
	that $G^{\otimes 1} =G$. By Theorem \ref{thm3} it is enough to show that the derivatives~$D_G(G^{\otimes k})$ and 
	$D_{G^{\otimes k}}(G)$ are finitely generated for each $1\leq k \leq n$. In addition, since $G$ and $\gamma_{n+1}(G)$ 
	are finitely generated, $\gamma_k(G)$ is also finitely generated for each $1\leq k \leq n+1$. 
	Moreover, we have that $D_{G^{\otimes k}}(G) = \gamma_{k+1}(G)$, so we only need to show that $D_G(G^{\otimes k})$ is finitely generated 
	for each $1\leq k \leq n$.
	
	Then, to see that 
	$D_G(G^{\otimes k})$ is finitely generated for each $1\leq k \leq n$ we will show that $D_G^{l}(G^{\otimes k})$ is finitely generated for each $k$ and $l$, such that $1\leq k, l \leq n$ and 
	$k+l \leq n+1$. 
	
	If $k=1$, then this is immediate. Assume that the same is true for a fixed integer $k$. Using Lemma~\ref{L:tech} we have that
	\[
	D_G^l (G^{\otimes k+1})=D_G^l (G^{\otimes k}\otimes G) = 
	\Img \Big\{  D_G^l (G^{\otimes k}) \otimes G \to G^{\otimes k} \otimes G \Big\}.
	\]
	If $k+l\leq n$, then by Theorem~\ref{thm3}, $D_G^l (G^{\otimes k}) \otimes G$ will be finitely generated, because~$D_G(D_G^l (G^{\otimes k}))= D_G^{l+1}(G^{\otimes k})$ and $D_{D^l (G^{\otimes k})}(G)=\gamma_{k+l+1}(G)$ are
	finitely generated. Thus, if $k+l\leq n$, we get that  $D_G^l (G^{\otimes k+1})$ is finitely generated as it is the homomorphic
	image of $D_G^l (G^{\otimes k}) \otimes G$. 
\end{proof}

\begin{corollary}
	If $G$ is a finitely generated perfect group, then $\gamma_{n+1}(H)$ is finitely generated for any extension of groups $1 \to N \to H \to G \to 1$ with $N\leq Z_n (H)$. 
\end{corollary}
Now we want to generalise Schur's Theorem, so we will introduce some notation. Let $G$ be a group, and let $G^{\wedge (1)} = G$. Then, we consider inductively
\[
G^{\wedge (n+1)}=G^{\wedge (n)}\wedge G^{\wedge (n)}, \quad n\geq 1. 
\]
We will call this notion the \emph{$n$-iterated exterior product}. For $n \geq 2$, there is a homomorphism $\mu_{n}^{G}\colon G^{\wedge (n)} \to G$, given by
$\mu_{n}^{G} =\mu^{G^{\wedge (1)}}\circ \cdots \circ \mu^{G^{\wedge (n-1)}}$, where 
$\mu^{G^{\wedge (i)} } \colon G^{\wedge (i)} \to G^{\wedge(i-1)}$ is the homomorphism introduced after Definition~\ref{D:exterior}.

For any $g_1, g_2, \ldots, g_{2^{k+1}} \in G$, we simplify the notation inductively as follows:
\[
g_1\wedge g_2\wedge\cdots \wedge g_{2^{k+1}}= (g_1\wedge g_2\wedge\cdots \wedge g_{2^{k}}) \wedge 
(g_{2^k+1}\wedge g_{2^k+2}\wedge \cdots \wedge g_{2^{k+1}}).  
\]

\begin{definition} For each $n\geq 1$, we define a group $\mathfrak{D}_n(G)$ by
	\[
	\mathfrak{D}_n(G)=\{ g\in G \mid [\cdots [[g, x_1], x_2], \cdots, x_{n}] =1 \;\text{for each} \; x_i\in \Gamma_i(G) \}.
	\]
	
	It is easy to observe that $\mu_{n+1}^G ( g_1\wedge g_2\wedge\cdots \wedge g_{2^{n+1}} )=1$ whenever
	$g_i\in \mathfrak{D}_n(G)$ for some $1\leq i \leq 2^{n+1}$. 
\end{definition}

\begin{lemma}\label{lemma2}  Let $1 \to N \to H \to G \to 1$ be an extension of groups such that ${N\leq \mathfrak{D}_n (H)}$ 
	for a fixed positive integer $n$. Then, there exists a surjective homomorphism $ G^{\wedge (n+1)} \to \Gamma_{n+1}(H)$ 
	such that the following diagram commutes: 
	\[
	\xymatrix{
		G^{\wedge (n+1)} \ar[r]^{\id}\ar[d] & G^{\wedge (n+1)} \ar[d]^{\mu_{n+1}^G}\\
		\Gamma_{n+1}(H) \ar[r]^{}&  G  } 
	\]
\end{lemma}
\begin{proof} By Theorem \ref{thm1}  we have an exact sequence of groups
	\[
	1 \to \overset{2^{n+1}}{\underset{i=1}{\prod}}N_i \to H^{\wedge (n+1)} \to G^{\wedge (n+1)} \to 1,
	\]
	where $N_i$ is the normal subgroup of $H^{\wedge (n+1)}$ generated by all elements of the form
	$h_1 \wedge h_2 \wedge \cdots \wedge h_{2^{n+1}}$ with $h_i\in N$ and $h_j\in H$ whenever $j\neq i$. 
	Moreover, we have the following commutative diagram:
	\[
	\xymatrix{
		H^{\wedge (n+1)} \ar[r]^{}\ar[d]_{\mu_{n+1}^H} & G^{\wedge (n+1)} \ar[d]^{\mu_{n+1}^G}\\
		H \ar[r]^{}&  G  } 
	\]
	For each $i\leq n+1$, we know that $\mu_{n+1}^{H} (N_i)=1$, then $\mu_{n+1}^{H} (H^{\wedge (n+1)})=\Gamma_{n+1}(H)$. Therefore, $\mu_{n+1}^{H} \colon H^{\wedge (n+1)} \to H$ induces an epimorphism $G^{\wedge (n+1)} \to \Gamma_{n+1}(H)$. 
\end{proof}

We now obtain another generalisation of Baer's Theorem.

\begin{proposition}\label{prop3} Let $1 \to N \to H \to G \to 1$ be an extension of groups such that $N\leq \mathfrak{D}_n (H)$ 
	for a fixed positive integer $n$.
	\begin{itemize}
		\item[\rm (i)]   if $G$ is finite (or a $p$-group), then $\Gamma_{n+1}(H)$ is finite (or a $p$-group);
		\item[\rm (ii)]  if $G$ is a perfect group, then  $\Gamma_{n+1}(H)$ is a perfect group;
		\item[\rm (iii)] if $G$ is polycyclic, then $\Gamma_{n+1}(H)$ is polycyclic.
	\end{itemize}
	
\end{proposition}
\begin{proof} By Lemma~\ref{lemma2} it suffices to prove that $G^{\wedge (n+1)}$ is a finite group (or a $p$-group, 
	or a perfect group, or a polycyclic group). The exterior product is, by definition, a quotient group of the non-abelian tensor square.
	Therefore, by Theorem~\ref{thm2} we get that $G^{\wedge (n+1)} $ is a finite group (or a $p$-group, or a perfect group, or a polycyclic group). 
\end{proof}

Note that Proposition~\ref{prop3}(i) was already shown in~\cite{Tur}.

\begin{proposition}\label{prop4} Let $n$ be a positive integer. Then the following are equivalent:
	\begin{itemize}
		\item[\rm (i)] $G^{\wedge (n+1)}$ is a finite group (or a $p$-group, or a perfect group, or a polycyclic group);
		\item[\rm (ii)] $\Gamma_{n+1}(H)$ is a finite group (or a $p$-group, or a perfect group, or a polycyclic group) for any extension of groups
		$1 \to N \to H \to G \to 1$ such that $N\leq \mathfrak{D}_n (H)$. 
	\end{itemize}
\end{proposition}
\begin{proof} The implication (i) $\implies$ (ii) follows directly from Lemma \ref{lemma2}. 
	Let us show the converse. Consider an extension of groups $1 \to R \to F \to G \to 1$ where $F$ is a free group. We denote 
	$\Gamma_1(R, F)=R$ and $\Gamma_{k+1}(R, F)=[\Gamma_k(R, F), \Gamma_k(F)]$, for $k\geq 1$. Then
	the following extension 
	\[
	1 \to \frac{R}{\Gamma_{n+1}(R, F)} \to  \frac{F}{\Gamma_{n+1}(R, F)} \to G \to 1
	\]
	is such that
	\[
	\frac {R}{\Gamma_{n+1}(R, F)} \leq \mathfrak{D}_n\Bigg ( \frac{F}{\Gamma_{n+1}(R, F)}\Bigg) . 
	\]
	Therefore, the group
	\[
	\Gamma_{n+1}\Bigg ( \frac{F}{\Gamma_{n+1}(R, F)}\Bigg) = \frac{\Gamma_{n+1}(F)}{\Gamma_{n+1}(R, F)}
	\]
	is in fact a finite group (or a $p$-group, or a perfect group, or a polycyclic group). Thus, we may complete the proof by showing that there is an isomorphism:
	\[
	G^{\wedge (n+1)} \cong \frac{\Gamma_{n+1}(F)}{\Gamma_{n+1}(R, F)}. 
	\]
	
	This is proven in~\cite[Theorem~2.1]{BlFuMo} for $n=1$. Let us assume that $G^{\wedge (n)} \cong \Gamma _{n}(F)/\Gamma_{n}(R, F)$ for $n\geq 2$. Since~$\Gamma _{n}(F)$ is a free group, 
	we have the following free presentation of $G^{(n)}$:
	\[
	1\to \Gamma_{n}(R, F) \to  \Gamma _{n}(F) \to G^{\wedge(n)} \to 1. 
	\]
	Therefore, using again~\cite[Theorem~2.1]{BlFuMo} we get
	\[
	G^{\wedge (n+1)}= G^{\wedge (n)}\wedge G^{\wedge (n)} \cong \frac{[ \Gamma _{n}(F),  \Gamma _{n}(F)]}{[ \Gamma_{n}(R, F), \Gamma _{n}(F)]} =  \frac{\Gamma _{n+1}(F)}{\Gamma_{n+1}(R, F)}.
	\]
\end{proof}

\begin{remark} \label{remark1} Let $1 \to R \to F \to G \to 1$ be an extension of groups such that $F$ is a free group.
	Lemma \ref{lemma2} tells us that there is an epimorphism 
	\[
	G^{\wedge (n+1)} \to \Gamma_{n+1}\bigg(\dfrac{F}{\Gamma_{n+1}(R, F)}\bigg).
	\]
	Following the proof of Proposition \ref{prop4} we can see that this map is in fact an isomorphism. 
\end{remark}

\

\section{Applications to non-abelian homological algebra}\label{S:Homological}

Let ${\bf Gr}$ be the category of groups, ${\bf Set}$ the category of sets and let $\mathcal U\colon {\bf Gr} \to {\bf Set}$ be the forgetful functor. The left adjoint functor of $\mathcal U$, is the functor $\mathcal F\colon {\bf Set} \to {\bf Gr}$ that assigns to a set $S$ the free group with basis $S$.
The pair of adjoint functors $(\mathcal{F}, \mathcal{U})$ induces a comonad $\mathcal P = (\mathcal P, \delta, \epsilon)$
on ${\bf Gr}$ in the usual way: $\mathcal P = \mathcal {FU} \colon {\bf Gr} \to {\bf Gr}$, $\epsilon \colon \mathcal P \to 1_{\bf Gr}$
is the counit and $\delta = \mathcal{F}\eta \mathcal{U}\colon \mathcal P \to \mathcal P ^2$ where 
$\eta \colon 1_{\bf Set} \to \mathcal{UF}$ is the unit of the adjunction. Given any group $G$ there is an augmented simplicial
group $\mathcal P _*(G) \to G$ where $\mathcal P _*(G) = \{\mathcal P _n(G), d^{i}_n, s^{i}_n\}$ is a simplicial group
defined as follows:
\[
\mathcal P _n (G) = \mathcal P ^{n+1} (G), \quad n\geq 0,
\]
\[
d^{i}_n =  \mathcal P ^i (\epsilon (\mathcal P ^{n-i}(G))), \quad s^{i}_n =  \mathcal P ^i (\delta (\mathcal P ^{n-i}(G))), \quad 0\leq i \leq n .
\]

Let $T\colon {\bf Gr} \to {\bf Gr}$ be a functor. As in~\cite{BB}, the left derived functors of $T$ with respect to the comonad 
$\mathcal P$ are given, for any group $G$, by
\[
\mathbf{L}_n T (G) = \pi _n (T(\mathcal P _*(G))), \; n\geq 0,
\]
where $T(\mathcal P _*(G))$ is the simplicial group obtained by applying the functor $T$ dimension-wise to $\mathcal P _*(G)$
and $\pi _n (T(\mathcal P _*(G)))$ is the $n$-th homotopy group of the simplicial group $T(\mathcal P _*(G))$. 
The functors $\mathbf{L}_n T $ may also be interpreted as non-abelian left derived functors of $T$, in the sense of
\cite{In}, relative to the projective class determined by the comonad $\mathcal P$. Furthermore, for any simplicial
group $F_*$ such that all the $F_n$ are free groups, $\pi_n (F_*)=1$ for $n\geq 1$ and $\pi_0 (F_*)=G$, we have:
\[
\mathbf{L}_n T (G) = \pi _n (T(F _*(G))), \quad n\geq 0.
\]
The simplicial group $F_*$ satisfying the aforementioned properties is called a free simplicial resolution of $G$. 

\begin{theorem}[\cite{BE, EGV}]\label{thm5}
	Let $1 \to R \to  F \to G \to 1$ be an extension of groups where $F$ is a free group.
	\begin{itemize}
		
		\item[\rm(i)] Let $T_k\colon {\bf Gr} \to {\bf Gr}$ be the functor given by $T_k (G)=G/\gamma_{k+1}(G)$, for $k\geq 1$, then
		\[
		\mathbf{L}_1 T_k (G)= \frac{R\cap \gamma_{k+1}(F)}{\gamma_{k+1}(R, F)},
		\]
		where $\gamma_1(R, F)=R$ and $\gamma_{n+1}(R, F)=[\gamma_n(R, F), F] = D_F(\gamma_{n}(R, F))$, for $n\geq 1$.
		
		\item[\rm(ii)] Let $T_{(k)} \colon {\bf Gr} \to {\bf Gr}$ be the functor given by $T_{(k)} (G)=G/\Gamma_{k+1}(G)$, for $k\geq 1$, then
		\[
		\mathbf{L}_1 T_{(k)} (G)= \frac{R\cap \Gamma_{k+1}(F)}{\Gamma_{k+1}(R, F)},
		\]
		where $\Gamma_1(R, F)=R$ and $\Gamma_{n+1}(R, F)=[\Gamma_n(R, F), \Gamma_n(F)]$, for $n\geq 1$.
	\end{itemize}
\end{theorem}

This theorem shows that $\mathbf{L}_1 T_1 (G)$ is, according to Hopf's formula isomorphic to the Schur multiplier of $G$, and $\mathbf{L}_1 T_k (G)$ is the Baer
invariant called the $k$-nilpotent multiplier of $G$, for $k\geq 2$. 

\begin{proposition}\label{prop5} Let $T_k$ and $T_{(k)}$ be the functors defined in the previous theorem, and recall $\lam_{k}^G$ and $\mu_{k}^G$ from Section~\ref{S:Baer}.
	\begin{itemize} 
		\item[\rm(i)]There is an epimorphism of groups 
		\[
		\Ker \big\{\lam_{k+1}^G \colon G^{\otimes k+1}\to  G \big\} \to \mathbf{L}_1 T_k (G).
		\]
		
		\item[\rm(ii)] There is an isomorphism of groups 
		\[
		\Ker \big\{\mu_{k+1}^G \colon G^{\wedge (k+1)}\to  G \big\} = \mathbf{L}_1 T_{(k)} (G).
		\]
	\end{itemize} 
\end{proposition}
\begin{proof} 
	The first part is a particular case of the exact sequence obtained in~\cite{BuEl}, as an adaptation of the main result of~\cite{Lue}.
	
	To prove the second part, we need to consider the following
	extension of groups:
	\[
	1 \to \frac{R}{\Gamma_{k+1}(R, F)} \to \frac{F }{\Gamma_{k+1}(R, F)} \to G \to 1 ,
	\]
	where $\Gamma_{k+1}(R, F)$ is defined as in Theorem \ref{thm5}. This extension satisfies the requirements of 
	Lemma \ref{lemma2} for $n=k$. Therefore, we have the following commutative diagram
	\[
	\xymatrix{
		G^{\wedge (k+1)} \ar[r]^{ \mu_{k+1}^G }\ar[d] & G \ar[d]^{1_G}\\
		\Gamma_{k+1}\Big (\frac {F}{\Gamma_{k+1}(R, F)}  \Big ) \ar[r]^{}&  G  } 
	\]
	In Remark~\ref{remark1} we had pointed out that the left vertical arrow in this diagram is an isomorphism. Hence,
	\[
	\Ker \mu_{k+1}^G = \Ker \Big \{ \Gamma_{k+1}\Bigg (\frac {F}{\Gamma_{k+1}(R, F)}  \Bigg ) \to  G \Big\} = 
	\Ker \Bigg \{ \frac {\Gamma_{k+1}(F)}{\Gamma_{k+1}(R, F)}  \to  G \Bigg\} =
	\]
	\[
	\frac {R\cap \Gamma_{k+1}(F)}{\Gamma_{k+1}(R, F)} \overset{\text{(by Theorem \ref{thm5}})}{=} 
	\mathbf{L}_1 T_{(k)} (G).
	\]
\end{proof}

\begin{corollary}
	If $G$ is a finite group (resp.\ a $p$-group), then $\mathbf{L}_1 T_{k} (G)$ and $\mathbf{L}_1 T_{(k)} (G)$ are
	finite groups (resp.\ $p$-group) for all $k\geq 1$. 
\end{corollary}

Let us recall that a group $G$ is said to be a Schur group if $G$ is perfect and the Schur multiplier of $G$
is trivial. 

\begin{theorem} 
	Let $G$ be a Schur group. Then, $\mathbf{L}_1 T_{k} (G)$ and $\mathbf{L}_1 T_{(k)} (G)$ are trivial for all $k\geq 1$. 
\end{theorem}
\begin{proof} It is well-known that the Schur multiplier of $G$ is isomorphic to the second homology group $H_2(G)$.
	By Theorem \ref{thm4} we get that $\mu^G \colon G\wedge G \to G$ is an isomorphism. This implies that 
	$\mu^{G^{\wedge (2)}}\colon G^{\wedge (2)}\wedge G ^{\wedge (2)}\to G^{\wedge (2)}$ is also an isomorphism, and by induction, we see that 
	$\mu^{G^{\wedge (n)}} \colon G^{\wedge (n)}\wedge G ^{\wedge (n)}\to G^{\wedge (n)}$ is an isomorphism for 
	all $n\geq 1$. Hence, $\mu_{k+1}^{G} =\mu^{G^{\wedge (1)}}\circ \cdots \circ \mu^{G^{\wedge (k)}}$
	is an isomorphism for $k\geq 1$. By Proposition \ref{prop5} we get $\mathbf{L}_1 T_{(k)} (G)=1$ for $k\geq 1$. 
	
	We will show that $\lam _{k+1}^G \colon G^{\otimes k+1} \to G$ is an isomorphism for $k\geq 1$. Since $G$ is a perfect group,
	the natural projection $G\otimes G \to G\wedge G$ is an isomorphism. This implies that $\lam_2^G \colon G\otimes G \to G$
	is an isomorphism of Schur groups. Therefore, it is easy to see that the map
	$\al_ n^G \colon G^{\otimes n+1} \to G^{\otimes n}$ given by 
	\[
	(\cdots ((g_1\otimes g_2)\otimes g_3) \otimes \cdots \otimes g_{n+1}) \mapsto 
	(\cdots ([g_1, g_2]\otimes g_3) \otimes \cdots \otimes g_{n+1}),
	\]
	extends to an isomorphism for all $n\geq 1$. Since $\lam_{k+1}^G =  \al_ {1}^G\circ \cdots \circ \al_ {k}^G$,
	by Proposition \ref{prop5} we get $\mathbf{L}_1 T_{k} (G)=1$ for $k\geq 1$.
\end{proof}

Let $G$ be a finitely generated group and $k$ be a fixed positive integer. Then the Baer invariant $\mathbf{L}_1 T_{k} (G)$ 
does not need to be finitely generated even when $\gamma_{n}(G)$ is finitely generated for all $n\geq 1$. But we can now prove 
the following:

\begin{theorem} Let $G$ be a finitely generated group and $k$ be a fixed positive integer such that 
	$\mathbf{L}_1 T_{k} (G)$ is finitely generated. Then, $\mathbf{L}_1 T_{n} (G)$ is finitely generated for all 
	$1\leq n\leq k$.
\end{theorem}
\begin{proof} Let $F_*$ be a free simplicial resolution of $G$. Since $G$ is finitely generated, we can assume that
	$F_0$ is a finitely generated free group. The natural morphism $T_{n+1}(F_*) \to T_{n}(F_*)$ is an epimorphism of
	simplicial groups for each $n\geq 1$. Let
	\[
	X_* = \Ker \big\{ T_{n+1}(F_*) \to T_{n}(F_*) \big\} . 
	\]
	Then, we have the following exact sequence of homotopy groups:
	\[
	\pi_1(T_{n+1}(F_*)) \to \pi_1 (T_{n}(F_*)) \to \pi_0 (X_*),
	\]
	which yields the following short exact sequence of groups:
	\[
	1\to A \to \mathbf{L}_1 T_{n} (G) \to B \to 1,
	\]
	where $A$ is a quotient subgroup of $\mathbf{L}_1 T_{n+1} (G)$ and $B$ is a subgroup of  $\pi_0 (X_*)$. Since $F_0$ is
	finitely generated, $X_0=\gamma_{n+1}(F_0)/\gamma_{n+2}(F_0)$ is a finitely generated abelian group. Therefore,
	$\pi_0 (X_*)$ is a finitely generated abelian group. Hence, $B$ is finitely generated. Thus, using the exact sequence given
	above we get that if $\mathbf{L}_1 T_{n+1} (G)$ is finitely generated, then  $\mathbf{L}_1 T_{n} (G)$ is also finitely
	generated.  
\end{proof}

\section*{Acknowledgements}
	This work was supported by Ministerio de Economía y Competitividad (Spain), with grant number MTM2016-79661-P. The second author is a Postdoctoral Fellow of the Research Foundation–Flanders (FWO).
	
	We would kindly like to thank the referees for their helpful suggestions which enabled us to improve the manuscript.

\end{document}